\documentclass[10pt,a4paper]{amsart}
\usepackage{a4wide}
\usepackage{amsfonts}
\usepackage{amsmath,latexsym,amssymb,amsfonts,amsbsy, amsthm, mathtools}
\usepackage[usenames]{color}
\usepackage{xspace,colortbl}
\allowdisplaybreaks

\newcommand{\beq}{\begin{equation}}
\newcommand{\eeq}{\end{equation}}
\newcommand{\ben}{\begin{eqnarray}}
\newcommand{\een}{\end{eqnarray}}
\newcommand{\beno}{\begin{eqnarray*}}
\newcommand{\eeno}{\end{eqnarray*}}

\theoremstyle{plain}
\newtheorem{theorem}{Theorem}[section]

\newtheorem{proposition}[theorem]{Proposition}
\newtheorem{lemma}[theorem]{Lemma}
\newtheorem{remark}[theorem]{Remark}

\numberwithin{theorem}{section} \numberwithin{equation}{section}
\renewcommand{\theequation}{\thesection.\arabic{equation}}

\newcommand{\average}{{\mathchoice {\kern1ex\vcenter{\hrule height.4pt
width 6pt depth0pt} \kern-9.7pt} {\kern1ex\vcenter{\hrule height.4pt
width 4.3pt depth0pt} \kern-7pt} {} {} }}

\newtheorem{thm}{Theorem}[section]

\newtheorem*{rmks*}{Remarks}
\newtheorem*{rmk*}{Remark}

\def\R{\mathbb{R}}

\renewcommand{\phi}{\varphi}

\newcommand{\be}{\begin{equation}}
\newcommand{\ee}{\end{equation}}

\newcommand{\C}{\mathbb{C}}

\newcommand{\N}{\mathbb{N}}

\newcommand{\cS}{{\mathcal S}}

\newcommand{\weak}{\rightharpoonup}

\newcommand{\eps}{\varepsilon}

\renewcommand{\theequation}{\thesection.\arabic{equation}}

\renewcommand{\epsilon}{\varepsilon}

\newcommand{\qq}{\mathsf{q}}
\newcommand{\CC}{\mathsf{C}}

\begin{document}
\title[Symmetry Breaking for Ground States of Biharmonic NLS]{Symmetry breaking for ground states of \\ biharmonic NLS via Fourier extension estimates}
\author{Enno Lenzmann}
\address{Enno Lenzmann, Departement Mathematik und Informatik, Universit\"at Basel, Spiegelgasse 1, CH-4051 Basel, Switzerland}%
\email{enno.lenzmann@unibas.ch}

\author{Tobias Weth}
\address{Tobias Weth, Institut f\"ur Mathematik, Goethe-Universit\"at Frankfurt, Robert-Mayer-Str.~10, D-60629 Frankfurt am Main, Germany}
\email{weth@math.uni-frankfurt.de}
\maketitle

\begin{abstract}
We consider ground states solutions $u \in H^2(\R^N)$ of biharmonic (fourth-order) nonlinear Schr\"odinger equations of the form
$$
\Delta^2 u + 2a \Delta u + b u - |u|^{p-2} u = 0 \quad \mbox{in $\R^N$}
$$
with positive constants $a, b > 0$ and exponents $2 < p < 2^*$, where $2^* = \frac{2N}{N-4}$ if $N > 4$ and $2^* = \infty$ if $N \leq 4$. By exploiting a connection to the adjoint Stein--Tomas inequality on the unit sphere and by using trial functions due to Knapp, we prove a general symmetry breaking result by showing that all ground states $u\in H^2(\R^N)$ in dimension $N \geq 2$ fail to be radially symmetric for all exponents $2 < p < \frac{2N+2}{N-1}$ in a suitable regime of $a,b>0$. 

As applications of our main result, we also prove symmetry breaking for a minimization problem with constrained $L^2$-mass and for a related problem on the unit ball in $\R^N$ subject to Dirichlet boundary conditions. 
\end{abstract}

\vskip 0.2cm

\renewcommand{\theequation}{\thesection.\arabic{equation}}
\setcounter{equation}{0}
\section{Introduction}

The study of biharmonic (fourth-order) nonlinear Schr\"odinger equations (NLS) has attracted a significant amount of attention in the recent past; see e.\,g.~\cite{KaSh-00, BaKoSa-00, FiIlPa-02, Pa-07, Pa-09, Pa-13, BoLe-17, BoCaDsNa-18, BoCaGoJe-19, BoCaGoJe-19b, FeJeMaMa-20}. The main purpose of the present paper is to prove a \textbf{symmetry breaking result} for ground state solutions  of biharmonic NLS (suitably defined as energy minimizers), which is in striking contrast to the well-known results of radial symmetry for ground states of classical second-order NLS. As a key ingredient in our approach to show symmetry breaking, we shall exploit a close connection between Fourier extension estimates and ground states for suitable biharmonic NLS. 

As a concrete model problem, we consider ground state solutions $u \in H^2(\R^N)$ of biharmonic NLS of the form
\begin{equation}
  \label{eq:biharmonic-general-a-b}
\Delta^2 u + 2 a \Delta u +b u - |u|^{p-2}u = 0 \qquad \text{in $\R^N$,}
\end{equation}
where $a, b > 0$ are positive constants. We remark that positivity of $a >0$ implies that the Fourier symbol of the `mixed dispersion' differential operator $\Delta^2 + 2 a \Delta$ is radially symmetric but fails to be monotone increasing in the radial variable. Throughout the following, we consider the subcritical case with $2<p<2^*$, where we set 
$$
2^* := \begin{dcases*} \frac{2N}{N-4} & for $N > 4$, \\
\infty & for $N \leq 4$ \end{dcases*}
$$ 
Furthermore, the constants $a,b >0$  in \eqref{eq:biharmonic-general-a-b} are chosen such that the associated quadratic form 
$$
\qq_{a,b}(u) = \int_{\R^N}\Bigl( |\Delta u|^2- 2 a  |\nabla u|^2 + b |u|^2 \Bigr)\,dx  = \int_{\R^N} g_{a,b}(|\xi|)| \hat u(\xi)|^2 \,d\xi
$$
is positive definite on the Sobolev space $H:= H^2(\R^N)$. Here $\hat{u}$ denotes the Fourier transform of $u$ and the corresponding symbol $g_{a,b}(|\xi|)$ above reads
$$
g_{a,b}(|\xi|)= |\xi|^4 - 2 a |\xi|^2 + b = (|\xi|^2-a)^2 + b-a^2.
$$ 
Thus we readily see that $\qq_{a,b}$ is positive definite if and only if $b>a^2$ holds. In this situation, we find that 
\begin{equation}
  \label{eq:def-kappa-a-b}
R_{a,b}(p):= \inf_{u \in H \setminus \{0\} } \frac{\qq_{a,b}(u)}{ \|u\|_p^2 } =\inf_{u \in H, \|u\|_p = 1}\qq_{a,b}(u)>0,
\end{equation}
and that this infimum is attained. Here $\|\cdot\|_p$ denotes the usual $L^p(\R^N)$-norm. More precisely, the following result follows from classical arguments.

\begin{thm}
  \label{sec:prop-intro}
  Suppose $N \geq 1, a > 0, b > 0$, and $2 <p < 2^*$. Then we have $R_{a,b}(p)>0$ if and only if $b>a^2$. Moreover, if $b>a^2$, then
  $R_{a,b}(p)$ is attained in $H \setminus \{0\}$, and every minimizer $u \in H \setminus \{0\}$ corresponds, after multiplication with a positive factor, to a solution of (\ref{eq:biharmonic-general-a-b}). Finally, any minimizer $u \in H \setminus \{0\}$ is real-valued up to a trivial constant complex phase, i.\,e., we have $e^{i \theta} u(x) \in \R$ for a.\,e.~$x \in \R$ and some $\theta \in \R$. 
\end{thm}

For the convenience of the reader, we provide a short proof of Theorem \ref{sec:prop-intro} in the appendix. Moreover, we note that a slightly different proof of the first statement in this theorem is given in \cite[Theorem 3.6]{FeJeMaMa-20}. We say that a solution $u \in H \setminus \{0\}$ of \eqref{eq:biharmonic-general-a-b} is a \textbf{ground state solution} if the infimum $R_{a,b}(p)$ in (\ref{eq:def-kappa-a-b}) is attained at the function $u$. To justify this notion, we recall that the energy functional associated with (\ref{eq:biharmonic}) is given by 
$$
E_{a,b}(u)= \frac{1}{2} \qq_{a,b}(u) - \frac{1}{p} \|u\|_p^p.
$$
A standard argument shows that the least energy value among nontrivial solutions of (\ref{eq:biharmonic-general-a-b}) is characterized as 
$$
c_{a,b} := \inf_{u \in H \setminus \{0\}} \sup_{t \ge 0}E_{a,b} (t u) = \left ( R_{a,b}(p) \right ) ^{\frac{p}{p-2}}
$$
Moreover, minimizers of the quotient in (\ref{eq:def-kappa-a-b}) correspond, up to multiplication by a positive factor, to nontrivial solutions of (\ref{eq:biharmonic-general-a-b}) where the least energy value $c_{a,b}$ is attained. This justifies the term ground state solutions. 

The main aim of the present paper is to study the asymptotics of $R_{a,b}(p)$ and the shape of ground state solutions in the limiting case where $b$ is slightly larger than $a^2$. Without loss of generality, by rescaling, we may assume that $a = 1$ from now on. By writing $b= 1+ \eps$ with some $\eps>0$, we thus arrive at the equation
\begin{equation}
  \label{eq:biharmonic}
\Delta^2 u + 2  \Delta u +(1+\eps) u - |u|^{p-2}u =0 \qquad \text{in $\R^N$.}
\end{equation}
Clearly, the corresponding quadratic is given by
\begin{equation}
  \label{eq:def-q-eps}
  \qq_\eps(u) = \int_{\R^N}\Bigl(|\Delta u|^2 - 2|\nabla u|^2 +(1+\eps)|u|^2\Bigr)\,dx = \int_{\R^N}g_\eps(|\xi|)|\hat u(\xi)|^2\,d\xi 
\end{equation}
where the Fourier symbol is given by
$$
g_\eps(|\xi|) = (|\xi|^2-1)^2 + \eps.
$$
Furthermore, the corresponding minimal energy quotient reads
\begin{equation}
  \label{eq:def-kappa-eps}
R_\eps(p):= \inf_{u \in H \setminus \{0\} } \frac{\qq_\eps(u) }{ \|u\|_p^2 }>0 \qquad \text{for $2<p<2^*$.}
\end{equation}
As noted above,  minimizers of the quotient in (\ref{eq:def-kappa-eps}) correspond, after multiplication by a positive factor, to ground state solutions of (\ref{eq:biharmonic}). We shall see below that, for fixed $p \in (2,2^*)$, the value $R_\eps(p)$ tends to zero as $\eps \to 0^+$. In fact, a detailed analysis of this limit shows that ground states $u \in H$ \emph{cannot} be radially symmetric for sufficiently small $\eps > 0$ and for exponents  
$$
2 < p < 2_*.
$$  
Here we define the exponent
$$
2_* := \frac{2N+2}{N-1},
$$
which arises from the adjoint version of the celebrated \textbf{Stein--Tomas inequality} for the Fourier restriction on the sphere $S^{N-1} \subset \R^N$ in dimension $N \geq 2$. More precisely, we have the following main result.
\begin{thm}[Symmetry Breaking]
  \label{theorem-nonradial}
If $N \ge 2$ and $2< p< 2_*$, then there exists $\eps_0= \eps_0(p)>0$ with the property that
  every ground state solution $u \in H$ of (\ref{eq:biharmonic}) is a nonradial function if $0<\eps \le \eps_0$.
\end{thm}

\begin{rmks*}
  {\em 1) It is interesting to note that the nonradial ground states $u \in H$ in Theorem \ref{theorem-nonradial} can still be {\em even} functions. By the recently developed Fourier symmetrization methods in \cite{BuLeSo-2019, LeSo-21}, one can prove for $N \ge 1$ and $p \in 2 \mathbb{N}$ with $2 < p < 2^*$ that any ground $u \in H$ for all $\eps > 0$ must be an even function (up to a translation in space), i.\,e., we have $u(-x) = u(x)$ for a.\,e.~$x \in \R^N$; see Lemma~\ref{lem:even} below. In particular, this result applies in the case $N=2$ and $p=4<2_*$ which is admissible for Theorem~\ref{theorem-nonradial}. So in this case, ground state solutions are nonradial but even up to translation.

2) The evenness result given in Lemma~\ref{lem:even} below also shows that Theorem~\ref{theorem-nonradial} cannot be extended to the case $N=1$ since symmetry breaking does not occur for $p \in 2 \mathbb{N}$.
    
3) We do not expect any symmetry breaking of ground states $u \in H^2(\R^N)$ for \eqref{eq:biharmonic-general-a-b} if $a \leq 0$ holds. In this case, the corresponding Fourier symbol $g_{a,b}(|\xi|)$ becomes strictly increasing in $|\xi|$ and, by Fourier rearrangement methods from \cite{LeSo-21}, we obtain radial symmetry (up to translation) for any ground state $u \in H^2(\R^d)$ provided that $p \in 2 \mathbb{N}$ is an even integer.  Also, by maximum principles and classical rearrangement techniques, it can be shown for any $2 < p < 2^*$  that ground states $u \in H^2(\R^N)$ for \eqref{eq:biharmonic-general-a-b} must be radial (up to translation) whenever $a < 0$ satisfies $|a| > \sqrt{b}$; see \cite[Theorem 3.9]{BoCaDsNa-18}.}
\end{rmks*}

  Note that Theorem~\ref{theorem-nonradial} does not rule out symmetry breaking for exponents $p \ge 2_*$. Nevertheless, the special role of the exponent $2_*$  is highlighted by the following result, which shows that the rate of convergence of the minimal energy quotient $R_\eps(p)$ in (\ref{eq:def-kappa-eps}) is $p$-independent if $2_* \le p < 2^*$, whereas it depends nontrivially on $p$ if $2<p<2_*$.

  \begin{thm}[Expansion of $R_\eps(p)$]
  \label{main-theo-intro-eps-expansion}
  Let $N \ge 2$. For $2 < p < 2^*$, there exist constants $\CC(p)>0$ with the following properties. 
  \begin{enumerate}
  \item[(i)] If $p \ge 2_*$, we have
   \begin{equation}
     \label{eq:asymptotic-expansion-r-eps}
     R_\eps(p)= \CC(p)\sqrt{\eps} + o \bigl(\sqrt{\eps}\bigr) \qquad \text{as $\eps \to 0^+$.}
   \end{equation}
  \item[(ii)] If $2 < p <2_*$, we have 
      \begin{equation}
     \label{eq:asymptotic-lower-bound}
    R_\eps(p) \ge \CC(p) \eps^{\frac{3}{4}+ \frac{1}{2p} -\frac{N}{2}(\frac{1}{2}-\frac{1}{p})} +o\bigl(\eps^{\frac{3}{4}+ \frac{1}{2p} -\frac{N}{2}(\frac{1}{2}-\frac{1}{p})}\bigr) \qquad \text{as $\eps \to 0^+$.}
   \end{equation}
   Moreover, this asymptotic lower bound is sharp in the sense that
   \begin{equation}
     \label{eq:asymptotic-upper-bound}
     R_\eps(p) = O(\eps^{\frac{3}{4}+ \frac{1}{2p} -\frac{N}{2}(\frac{1}{2}-\frac{1}{p})}) \qquad  \text{as $\eps \to 0^+$.}
    \end{equation}
  \end{enumerate}
\end{thm}

\begin{rmks*} {\em 
    1) The constants $\CC(p)$ are characterized in (\ref{eq:characterization-C-p}) below. 
    
    2) Note that $\frac{3}{4}+ \frac{1}{2p} -\frac{N}{2}(\frac{1}{2}-\frac{1}{p}) = \frac{1}{2}$ for $p=2_*$ holds. Thus the dependence on the exponent on $p$ is continuous. 

3) In \cite{FeJeMaMa-20}[Proposition 3.7], the authors derive the general upper bound $R_\eps(p) \leq C \sqrt{\eps}$ for $\eps > 0$ sufficiently small and some constant $C > 0$ depending on $p$ and $N$. However, such a bound will not be sufficient to conclude the symmetry breaking result in Theorem \ref{theorem-nonradial}.
 }
\end{rmks*}

Next, we discuss an application of Theorem~\ref{theorem-nonradial} to the associated energy minimization problem with fixed mass (i.e., $L^2$-norm), which has been studied recently in \cite{FeJeMaMa-20}. For this we consider the energy functional
\begin{equation}
  \label{eq:def-tilde-E}
  \tilde E: H \to \R, \qquad \tilde E(u) = \int_{\R^N} |\Delta u|^2\,dx - 2 \int_{\R^N} |\nabla u|^2\,dx - \frac{2}{p} \int_{\R^N} |u|^p\,dx 
\end{equation}
and the fixed mass constraint given by  
\begin{equation}
  \label{eq:def-s-m}
S(m):= \Bigl \{ u \in H\::\: \int_{\R^N}|u|^2\,dx = m \Bigr\}.
\end{equation}
As discussed in detail in \cite{FeJeMaMa-20}, both the energy $\tilde E$ and the set $S(m)$ are invariant under the corresponding biharmonic nonlinear Schr\"odinger flow. As a consequence of this invariance, the problem of minimizing $\tilde E$ on $S(m)$ is closely related to orbital stability properties of the set of associated minimizers. In \cite[Theorems 1.2 and 1.3]{FeJeMaMa-20}, it is proved that, for every $m>0$, the infimum of $\tilde E$ on $S(m)$ is attained in the mass-subcritical case where
$2<p<\max(4,\frac{2(N+5)}{N+1})$ and $p<2 + \frac{8}{N}$, and every minimizer $u \in S(m)$ is a ground state solution of (\ref{eq:biharmonic}) for some $\eps= \eps(m)$, whereas $\eps(m) \to 0^+$ as $m \to 0$. The following theorem on symmetry breaking in the case of small fixed mass is an immediate corollary of these results and Theorem~\ref{theorem-nonradial}.

\begin{thm}
  \label{theorem-nonradial-fixed-norm}
Let $N \ge 2$, and suppose that $2<p<\frac{14}{3}$ if $N=2$ and $2 <p< 2_*$ if $N \ge 3$. Then there exists $m_0= m_0(p)>0$ with the property that for every $0<m<m_0(p)$ all minimizers of $\tilde E$ on $S(m)$ are nonradial functions.
\end{thm}

In our final main result, we show that the symmetry breaking phenomenon is not restricted to biharmonic equations in the entire space but also arises in the case of Dirichlet problems in the unit ball $B=B_1(0)$. More precisely, we consider the boundary value problem 
\begin{equation}
  \label{eq:biharmonic-bvp-ball}
\left\{  \begin{aligned}
    &\Delta^2 u + 2 a \Delta u +b u - |u|^{p-2}u = 0 &&\qquad \text{in $B$,}\\
    &u= \partial_\nu u =0 &&\qquad \text{on $\partial B$.} 
  \end{aligned}
\right.  
\end{equation}
Related to (\ref{eq:biharmonic-bvp-ball}) we consider the restriction 
$$
u \mapsto \qq_{a,b,B}(u) = \int_{B}\Bigl( |\Delta u|^2- 2 a  |\nabla u|^2 + b |u|^2 \Bigr)\,dx 
$$
of the quadratic form $\qq_{a,b}$ to the subspace $H_0^2(B) \subset H$ and the value 
\begin{equation}
  \label{eq:def-kappa-a-b-ball}
R_{a,b,B}(p):= \inf_{u \in H^2_0(B) \setminus \{0\} } \frac{\qq_{a,b,B}(u)}{ \|u\|_{L^p(B)}^2 }.
\end{equation}
Similarly as in Theorem~\ref{sec:prop-intro}, we see that, for $N \geq 1, a > 0$ $b \ge a^2$ and $2 <p < 2^*$ we have $R_{a,b,B}(p)>0$, and this value is attained in $H^2_0(B) \setminus \{0\}$. Moreover, every minimizer $u \in H^2_0(B) \setminus \{0\}$ corresponds, after multiplication with a positive factor, to a solution of (\ref{eq:biharmonic-bvp-ball}).
We say that a solution $u \in H^2_0(B) \setminus \{0\}$ of (\ref{eq:biharmonic-bvp-ball}) is a ground state solution if the infimum in (\ref{eq:def-kappa-a-b-ball}) is attained at $u$. We then have the following result.

\begin{thm}
  \label{theorem-nonradial-ball}
Let $N \ge 2$ and $2< p< 2_*$, and let $\eps_0= \eps_0(p)>0$ be given by Theorem~\ref{theorem-nonradial}. For $0<\eps \le \eps_0$, there exists $a_0=a_0(\eps,p)>0$ with the property that
every ground state solution $u \in H$ of (\ref{eq:biharmonic-bvp-ball}) is a nonradial function if $a>a_0$ and $b=(1+\eps)a^2$. 
\end{thm}

Up to our knowledge, this is the first result on nonradiality of ground state solutions for a rotationally invariant semilinear elliptic Dirichlet problem with constant coefficients in a ball. In the case $a=b=0$ and $2< p< 2_*$, the radial symmetry (and uniqueness) of ground state solutions of (\ref{eq:biharmonic-bvp-ball}) has been proven in \cite{Ferrero-Gazzola-Weth}, whereas in the remaining cases the question of radial symmetry remains largely open. Related to this aspect, we mention the analogue in \cite{berchio-gazzola-weth} of the Gidas--Ni--Nirenberg result on Schwarz symmetry of nonnegative solutions for polyharmonic Dirichlet problems in the unit ball with increasing nonlinearity and the counterexamples given in \cite{sweers} and \cite{gazzola-sperone}. For a comprehensive discussion of various aspects of semilinear higher order boundary value problems, see \cite{gazzola-grunau-sweers}.

As already indicated above, the limiting exponent $2_*$ in Theorem~\ref{main-theo-intro-eps-expansion} hints at the Stein--Tomas inequality (see \cite{To-1975,St-1986}), which indeed will be of key importance in our paper. We recall this inequality in the following convenient adjoint version as a Fourier extension estimate. 

\begin{thm}[Stein--Tomas Inequality, Adjoint Version]
\label{stein-tomas}
Suppose $N \geq 2$ and let $S:= S^{N-1}$ be the unit sphere in $\R^N$. If $p \ge 2_*$, then 
\begin{equation}
\label{stein-tomas-quotient}
\CC_{ST}(p):= \inf_{w \in L^2(S) \setminus \{0\}} \frac{\|w\|_{L^2(S)}^2}{\|\check w\|_p^2} >0,
\end{equation}
where, for $w \in L^2(S)$, the function $\check w \in L^p(\R^N)$ is a.\,e. given by  
$$
\check w(x)= (2\pi)^{-N/2} \int_{S}e^{ix \cdot \theta}w(\theta) \, d\sigma(\theta).
$$
Consequently, the inequality $\|\check w\|_p \le \frac{1}{\sqrt{\CC_{ST}(p)}} \|w\|_{L^2(S)}$ holds for every $w \in L^2(S)$.
\end{thm}

The existence of optimizers for the Stein--Tomas inequality above in the non-endpoint case when $p> 2_*$ can be inferred from \cite{FaVeVi-2011}. However, existence of optimizers in the endpoint case $p=2_*$ is an open problem except for the cases $N \in \{2,3\}$ (see \cite{ChSh-2012,Sh-2016,Fo-2015}); see also \cite{FrLiSa-2016} for a conditional existence result for the endpoint case in general space dimensions.

The Stein--Tomas inequality plays a key role in the proof of the expansions given in Theorem \ref{main-theo-intro-eps-expansion}. In fact, we obtain the following characterization of the constants $\CC(p)$ occurring in Theorem \ref{main-theo-intro-eps-expansion} in terms of the constants $\CC_{ST}(p)$ in the Stein--Tomas inequality:
\begin{equation}
  \label{eq:characterization-C-p}
\CC(p)= \begin{dcases*} \frac{2}{\pi} \CC_{ST}(p) & if $2_* \le p < 2^*$, \\  \Bigl(\frac{2}{\pi} \CC_{ST}(2_*) \Bigr)^{(N+1)(\frac{1}{2}-\frac{1}{p})} & if $2<p<2_*$. \end{dcases*}  
\end{equation}
While the Stein--Tomas inequality is sufficient to derive the asymptotic expansion (\ref{eq:asymptotic-expansion-r-eps}) in the case where $p \ge 2_*$, we have to combine the Stein--Tomas inequality with interpolation estimates to obtain the lower asymptotic bound (\ref{eq:asymptotic-lower-bound}) in the case $2<p<2_*$. It is somewhat surprising that this approach already yields the optimal exponent, as shown by (\ref{eq:asymptotic-upper-bound}). To obtain the sharp asymptotic upper bound (\ref{eq:asymptotic-upper-bound}), we have to construct suitable nonradial test functions to estimate the quantity $R_\eps(p)$. The construction builts on the well-known test functions used by Knapp to characterize the optimal exponent $2_*$ for the Stein-Tomas inequality, see e.g. \cite[Chapter 7]{wolff}. For the attentive reader, we mention that the numerical factor $\frac{2}{\pi}$ in (\ref{eq:characterization-C-p}) appears due to the second-order derivative of the Fourier symbol $g_{\eps}$ at its minimum; see Lemma \ref{lemma-appendix} below.  

The paper is organized as follows.  In Section~\ref{sec:symm-break-proof}, we first derive Theorem~\ref{theorem-nonradial} from Theorem~\ref{main-theo-intro-eps-expansion} and related asymptotic bounds for radial functions, see Theorem~\ref{main-theo-intro-eps-expansion-rad} below. In Sections~\ref{sec:upper-estim-kapp} and \ref{sec:lower-estim-kapp}, we then complete the proof of Theorem~\ref{main-theo-intro-eps-expansion} by deriving upper and lower asymptotic bounds for the quantity $R_\eps(p)$. Moreover, we prove the radial asymptotic estimates given in Theorem~\ref{main-theo-intro-eps-expansion-rad}. In Section~\ref{sec:dirichl-probl-unit}, we then consider the Dirichlet problem~(\ref{eq:biharmonic-bvp-ball}) and complete the proof of Theorem~\ref{theorem-nonradial-ball}. Finally, in Section~\ref{sec:technical-lemma} we prove an elementary technical lemma which is needed in the proofs of Theorems~\ref{main-theo-intro-eps-expansion} and~\ref{main-theo-intro-eps-expansion-rad}, and in Section~\ref{sec:exist-prop-ground} we prove Theorem~\ref{sec:prop-intro}.

\section{Symmetry Breaking: Proof of Theorem \ref{theorem-nonradial}}
\label{sec:symm-break-proof}

We first prove the symmetry breaking result stated in Theorem \ref{theorem-nonradial}, which will be based on Theorem \ref{main-theo-intro-eps-expansion} and further estimates, whose proofs will be postponed to the sections below. Let $H_{rad}$ denote the closed subspace of radial functions in $H=H^2(\R^N)$. For $2 < p < 2^*$, we define
\begin{equation}
  \label{eq:def-kappa-eps-rad}
R^{rad}_{\eps}(p):= \inf_{u \in H_{rad} \setminus \{0\} } \frac{\qq_\eps(u) }{ \|u\|_p^2 } \:\ge \: R_\eps(p).
\end{equation}
Our goal is to show that, for exponents $2<p < 2_*$, there exists $\eps_0= \eps_0(p)>0$ with the property that
\begin{equation}
  \label{eq:rad-strict-ineq}
R^{rad}_\eps(p) > R_\eps(p) \qquad \text{for $0 < \eps < \eps_0$.}
\end{equation}
Once this is proved, it immediately follows that all ground state solutions of \eqref{eq:rad-strict-ineq} are nonradial for $0< \eps < \eps_0$. The starting point of the proof of \eqref{eq:rad-strict-ineq} is the well-known observation that the range of admissible exponents in the Stein--Tomas inequality can be extended for the subspace of radial functions in $L^2(S)$, which is merely the one-dimensional space of constant functions defined on $S = S^{N-1}$. For this we recall that the function $\check {1_S} \in C^\infty(\R^N)$ is given by
  $$
 \check{1_S} (x) = (2\pi)^{-N/2} \int_{S}e^{ix \cdot \theta}d\sigma(\theta), \qquad x \in \R^N.
  $$
 By standard estimates for oscillatory integrals, we obtain the bound
  $$
  |\check{1_S} (x)| \le C_N(1+|x|)^{-\frac{N-1}{2}} \qquad \text{with some constant $C_N>0$.}
  $$
  As a consequence, we have
  $$
  \check {1_S} \in L^p(\R^N)\qquad \text{for $p > 2_*^{rad}:= \frac{2N}{N-1}$.}
  $$
  We therefore may define
\begin{equation}
\label{stein-tomas-quotient-rad}
\CC_{ST}^{rad}(p):= \frac{\|w_0\|_{L^2(S)}^2}{\|\check w_0\|_p^2} = \frac{\omega_{N-1}}{\|\check w_0\|_p^2},
\end{equation}
where $\omega_{N-1}$ denotes the measure of the unit sphere $S = S^{N-1} \subset \R^N$. Thus, for every radial (i.e., constant) function $w \in L^2(S)$, we have the inequality
$$
\|\check w\|_p \le \frac{1}{\sqrt{\CC_{ST}^{rad}(p)}} \|w\|_{L^2(S)}^2 \quad \text{for $p > 2_*^{rad}$.}
$$
Based on these estimates, we can prove the following asymptotic estimates for $R^{rad}_\eps(p)$.
\begin{thm}
  \label{main-theo-intro-eps-expansion-rad}
  Let $N \ge 2$. We have the following estimates.
  \begin{enumerate}
  \item[(i)] If $2_*^{rad} <p < 2^*$, we have 
   \begin{equation}
        \label{main-theo-intro-eps-expansion-eq-1-rad}
    R^{rad}_\eps(p)= \frac{2 \CC_{ST}^{rad}(p)}{\pi}\sqrt{\eps} + o \bigl(\sqrt{\eps}\bigr) \qquad \text{as $\eps \to 0^+$.}
    \end{equation}
  \item[(ii)] If $2 < p \le 2_*^{rad}$, then for every $\beta > 1- N(\frac{1}{2}-\frac{1}{p})$ there exists a constant $\CC(p,\beta)>0$ with the property that 
   \begin{equation}
        \label{main-theo-intro-eps-expansion-eq-2-rad}
    R^{rad}_\eps(p) \ge \CC(p,\beta)\eps^{\beta} +o\bigl(\eps^{\beta}\bigr) \qquad \text{as $\eps \to 0^+$.}
    \end{equation}
  \end{enumerate}
\end{thm}

\begin{proof}
See Sections  \ref{sec:upper-estim-kapp} and  \ref{sec:lower-estim-kapp} below.
\end{proof}

For $N \ge 2$ and $2 < p < 2_*$, the key strict inequality \eqref{eq:rad-strict-ineq} now follows by combining Theorem \ref{main-theo-intro-eps-expansion-rad} with the estimates in Theorem~\ref{main-theo-intro-eps-expansion}, since we have
$$
\frac{3}{4}+ \frac{1}{2p} -\frac{N}{2}(\frac{1}{2}-\frac{1}{p}) >
\left\{
  \begin{aligned}
    &\frac{1}{2} &&\qquad \text{for $2_*^{rad} <p < 2_*$},\\
    &1- N(\frac{1}{2}-\frac{1}{p})&&\qquad \text{for $2< p \le 2_*^{rad}$.}
  \end{aligned}
\right.
$$
This proves the that strict inequality \eqref{eq:rad-strict-ineq} holds for some $\eps_0=\eps_0(p) > 0$. This completes the proof of Theorem~\ref{theorem-nonradial}.  \hfill $\square$

\section{Upper estimates for $R_\eps(p)$ and $R_\eps^{rad}(p)$}
\label{sec:upper-estim-kapp}

In this section, we prove the upper estimates for $R_\eps(p)$ and $R_\eps^{rad}(p)$ needed in the proofs of Theorems \ref{main-theo-intro-eps-expansion} and \ref{main-theo-intro-eps-expansion-rad}. We begin with the following result.

\begin{proposition}
\label{prop:upper-estim-kapp-large}  
For $2_* \le p < 2^*$, we have 
$$
R_\eps(p) \le  \frac{2 \CC_{ST}(p)}{\pi} \sqrt{\eps} + o(\sqrt{\eps}) \qquad \text{as $\eps \to 0^+$},  
$$
$$
R_\eps^{rad}(p) \leq \frac{2 \CC_{ST}^{rad}(p)}{\pi} \sqrt{\eps} + o(\sqrt{\eps}) \qquad \text{as $\eps \to 0^+$}.
$$
\end{proposition}

\begin{proof}
Recall that $S = S^{N-1} \subset \R^N$ denotes the unit sphere. Assume first that $p  > 2_*$ and let $w \in L^2(S)$ be an extremal function for the adjoint Stein--Tomas inequality, i.\,e., 
$$
\|w\|_{L^2(S)}^2 = \CC_{ST}(p) \|\check w\|_{p}^2,
$$
where $\check w$ is given by 
$$
\check w(x)= \int_{S}e^{ix \theta} w(\theta) \, d \sigma(\theta).
$$
We then fix $s \in (0,\frac{1}{2})$, put 
$$
\rho_\eps := \int_{1-\eps^s}^{1+\eps^s} \frac{r^{N-1} \, dr}{g_{\eps} (r)} \qquad \text{for $\eps \in (0,1)$}
$$
and we define $u_\eps \in H$ by its Fourier transform
$$
\hat u_\eps(\xi):= \left \{
  \begin{aligned}
&\frac{1}{g_\eps(|\xi|)} w(\frac{\xi}{|\xi|}) &&\qquad \mbox{if $\bigl||\xi|-1\bigr| \le \eps^s$},\\
&0 && \qquad \mbox{if $ \bigl||\xi|-1\bigr| \ge \eps^s$}.
\end{aligned}
\right.
$$
Then we have 
\begin{align*}
\qq_\eps(u_\eps)&= \int_{\R^N} g_\eps(|\xi|) |\hat u_\eps(\xi)|^2d \xi= 
\int_{1-\eps^s}^{1+\eps^s} \frac{r^{N-1}}{g_\eps(r)}  \int_{S} |w(\theta)|^2 \, d \sigma( \theta) \, dr \\
&= \rho_\eps \int_{S} |w(\theta)|^2 \, d\sigma(\theta) =\rho_\eps \CC_{ST}(p) \|\check w\|_{p}^2. 
\end{align*}
Moreover, for $x \in \R^N$ we have 
\begin{align*}
u_\eps(x) = (2\pi)^{-N/2} \int_{\R^N} e^{ix \xi} \hat u_\eps(\xi) \, d\xi &= (2\pi)^{-N/2} \int_{1-\eps^s}^{1+\eps^s} \frac{r^{N-1}}{g_\eps(r)} \int_{S} 
e^{i r x \xi} w(\theta) d\sigma(\theta) \, dr\\
&= \int_{1-\eps^s}^{1+\eps^s} \frac{r^{N-1}}{g_\eps(r)} \check w(rx) \, dr.
\end{align*}
Therefore,
$$
|\frac{u_\eps(x)}{\rho_\eps} - \check w(x)| = \frac{1}{\rho_\eps} \int_{1-\eps^s}^{1+\eps^s} \frac{r^{N-1}}{g_\eps(r)} \Bigl[\check w(rx)-\check w(x)\Bigr] dr \le  \sup_{|r-1|\le \eps^s} |\check w(rx)-\check w(x)| \to 0 
$$
as $\eps \to 0^+$. Hence Fatou's Lemma yields 
$$
\liminf_{\eps \to 0^+} \left \|\frac{u_\eps(x)}{\rho_\eps} \right \|_p \ge \|\check w\|_p .
$$
Consequently, 
$$
R_\eps(p) \le \frac{\qq_\eps(u_\eps)}{\|u_\eps\|_p^2} \le \frac{\rho_\eps \CC_{ST}(p) \|\check w\|_{p}^2}{\rho_\eps^2 \bigl(\|\check w\|_{p}^2+ o(1)\bigr)}
\le \frac{\CC_{ST}(p)}{\rho_\eps} (1+o(1)) 
$$ 
as $\eps \to 0^+$. Moreover, since
$$
\rho_\eps = \bigl(1 + o(1)\bigr)\int_{1-\eps^s}^{1+\eps^s} \frac{dr}{g_{\eps}(r)} \qquad \text{as $\eps \to 0^+$,}
$$
it follows from Lemma~\ref{lemma-appendix} and Remark~\ref{remark-appendix} in Appendix A below that we have
$$
\sqrt{\eps}\rho_\eps \to \frac{\pi}{\sqrt{g_0''(1)}} = \frac{\pi}{2} \qquad \text{as $\eps \to 0^+$}.
$$
Thus we conclude
$$
\limsup_{\eps \to 0^+} \frac{R_\eps(p)}{\sqrt{\eps}} \le \CC_{ST}(p)
\lim_{\eps \to 0^+}\frac{1}{\sqrt{\eps} \rho_\eps}= \frac{2 \CC_{ST}(p)}{\pi}. 
$$
This proves the claimed upper bound for $R_\eps(p)$ in the case $2_* < p < 2^*$. In the endpoint case when $p=2_*$ (and the existence of optimizers $w \in L^2(S)$ for the Stein--Tomas inequality in dimensions $N \geq 4$ is still open), we can choose for any $\delta > 0$ an approximate optimizer $w \in L^2(S)$ such that $\| w \|_{L^2(S)}^2 = (\CC_{ST}(p_*) + \delta) \| \check{w} \|_p^2$. By the exact reasoning as above, we find that 
$$
R_\eps(p_*) \leq \frac{2 (\CC_{ST}(p_*) + \delta)}{\pi} \sqrt{\eps} + o(\eps).
$$
Since $\delta > 0$ can be chosen arbitrarily, we are done.

Finally, we remarks that the upper estimate for $R_\eps^{rad}(p)$ with $2_* \leq p < 2^*$ again follows by the above arguments if we take the constant function $w \equiv 1 \in L^2(S)$ on the unit sphere.
\end{proof}

Next, we treat the case $p > 2_*$.

\begin{proposition}
\label{prop:upper-estim-kapp}  
For $2 < p < 2_*$, we have 
$$
R_\eps(p) = O\bigl(\eps^{\frac{3}{4}+ \frac{1}{2p} -\frac{N}{2}(\frac{1}{2}-\frac{1}{p})}\bigr) \qquad \text{as $\eps \to 0^+$.}  
$$
\end{proposition}

\begin{proof}
  Inspired by Knapp's well known example (see e.g. \cite[Chapter 7]{wolff}), we construct test functions by using characteristic functions of spherical caps. Let $\eps  \in (0,1)$ in the following. As usual, we use $S=S^{N-1}$ to denote unit sphere in $\R^N$. We define the spherical cap
  \begin{equation}
    \label{eq:def-c-eps}
  C_\eps = \{\theta \in S\::\: 1 - \theta_N \le \eps^{\frac{1}{2}}\}= \{\theta \in S\::\: |\theta-e_N| \le \sqrt{2} \eps^{\frac{1}{4}}\},
  \end{equation}
  where the latter equality follows since $|\theta- e_N|^2 = 2(1- \theta \cdot e_N) = 2(1 - \theta_N)$ for $\theta \in S$.
  We note that 
  \begin{equation}
    \label{eq:theta-i-est}
  |\theta_i| \le |\theta-e_N| \le \sqrt{2}\eps^{\frac{1}{4}} \qquad \text{for $\theta \in C_\eps$, $i=1,\dots,N-1$,}
  \end{equation}
  and that 
  $$
  \|w_\eps\|_{L^2(S)}^2 =  |C_\eps|  \qquad \text{for $w_\eps := 1_{C_\eps} \in L^2(S)$.}
  $$
In the following, we shall estimate $R_\eps(p)$ with the test function $u_\eps \in H$ defined by 
$$
\hat u_\eps(\xi):= \left \{
  \begin{aligned}
&w_\eps(\frac{\xi}{|\xi|}) &&\qquad \text{if $\bigl||\xi|-1\bigr| \le \sqrt{\eps}$,}\\
&0 &&\qquad \text{if $\bigl||\xi|-1\bigr| \ge \sqrt{\eps}$.}   
\end{aligned}
\right.
$$
Since $0 \le g_\eps(r) \le C \eps$ for $|r-1| \le \sqrt{\eps}$ with a constant $C>0$, we have 
\begin{align}
\qq_\eps(u_\eps)&= \int_{\R^N} g_\eps(|\xi|) |\hat u_\eps(\xi)|^2d \xi \le C \eps  
           \int_{1-\sqrt{\eps}}^{1+\sqrt{\eps}} r^{N-1}  \int_{S} |w_\eps(\theta)|^2d\sigma(\theta) dr \nonumber\\
  &\le 2C \eps^{\frac{3}{2}}(1+\sqrt{\eps})^{N-1}\|w_\eps\|_{L^2(S)}^2 \le 2^N C|C_\eps|\eps^{\frac{3}{2}}.\label{q-eps-upper-est}
\end{align}
To estimate $\|u_\eps\|_p$, we now define, for $\delta>0$, the set 
$$
M_{\eps,\delta}:= \{x \in \R^N\::\: |x_N| \le \delta \eps^{-\frac{1}{2}},\; |x_i|\le \delta \eps^{-\frac{1}{4}}\: \text{for $i=1,\dots,N-1$.} \},
$$
which has the volume 
$$
|M_{\eps,\delta}| = \delta^N \eps^{-\frac{N+1}{4}}.
$$
We also note that 
\begin{equation}
  \label{eq:real-part-1}
{\textbf{Re}\,} \Bigl(e^{-i x_N} \check w_\eps (x)\Bigr) = (2\pi)^{-\frac{N}{2}} \int_{C_\eps} {\textbf{Re}\,} \bigl( e^{i x \cdot (\theta- e_N)}\bigr) d \sigma(\theta)
= (2\pi)^{-\frac{N}{2}}  \int_{C_\eps} \cos \bigl( x \cdot (\theta- e_N)\bigr) d \sigma(\theta), 
\end{equation}
whereas, by (\ref{eq:def-c-eps}) and (\ref{eq:theta-i-est}),
$$
x \cdot (\theta- e_N) = \sum_{i=1}^{N-1}x_i \theta_i + x_N (\theta_N- 1) \le [(N-1)\sqrt{2} + 1]\delta \le 2N \delta \qquad \text{for $x \in M_{\eps,\delta}$, $\theta \in C_\eps$.}
$$
Hence, setting $\delta_0 := \frac{\pi}{8N}$, we have
$$
|x \cdot (\theta- e_N)| \le \frac{\pi}{4} \qquad \text{for $x \in M_{\delta_0,\eps}$, $\theta \in C_\eps$,}
$$
By (\ref{eq:real-part-1}), we deduce that 
\begin{equation}
  \label{eq:real-lower-bound}
{\textbf{Re}\,} \Bigl(e^{-i x_N} \check w_\eps (x)\Bigr) \ge  (2\pi)^{-\frac{N}{2}} |C_\eps| \cos \frac{\pi}{4} = \frac{|C_\eps|}{\sqrt{2}(2\pi)^{\frac{N}{2}}} \qquad \text{for $x \in M_{\eps,\delta_0}$.}
\end{equation}
Similarly, we also compute that
\begin{align}
  \Bigl|{\textbf{Im}\,} \Bigl(e^{-i x_N} \check w_\eps (x)\Bigr)\Bigr| &= (2\pi)^{-\frac{N}{2}}  \Bigl| \int_{C_\eps} {\textbf{Im}\,} \bigl(e^{i x \cdot (\theta- e_N)}\bigr) d \sigma(\theta)\Bigr| \le (2\pi)^{-\frac{N}{2}} |C_\eps| \sin \frac{\pi}{4} \nonumber\\
 &=  \frac{|C_\eps|}{\sqrt{2}(2\pi)^{\frac{N}{2}}} \qquad \text{for $x \in M_{\eps,\delta_0}$.}\label{Im-modulus-bound}
\end{align}
For $x \in \R^N$ we now have 
$$
u_\eps(x) =  \int_{\R^N} e^{ix \xi} \hat u_\eps(\xi) d\xi = (2\pi)^{-N/2} \int_{1-\sqrt{\eps}}^{1+\sqrt{\eps}}r^{N-1} \int_{S} 
e^{i r x \theta} w_\eps(\theta) d\sigma(\theta) = \int_{1-\sqrt{\eps}}^{1+\sqrt{\eps}} r^{N-1}\check w_\eps(rx)dr
$$
and therefore
\begin{align}
  &{\textbf{Re}\,} \Bigl(e^{-i x_N} u_\eps(x)\Bigr) = \int_{1-\sqrt{\eps}}^{1+\sqrt{\eps}} r^{N-1} {\textbf{Re}\,} \Bigl(e^{-i x_N} \check w_\eps(rx)\Bigr)dr \nonumber\\
                                       &= \int_{1-\sqrt{\eps}}^{1+\sqrt{\eps}} r^{N-1} \Bigl[ {\textbf{Re}\,} \bigl(e^{-i r x_N} \check w_\eps(rx)\bigr){\textbf{Re}\,} e^{i(r-1)x_N} - {\textbf{Im}\,} \bigl(e^{-i r x_N} \check w_\eps(rx)\bigr){\textbf{Im}\,} e^{i(r-1)x_N}\Bigr)dr \nonumber\\
   &= \int_{1-\sqrt{\eps}}^{1+\sqrt{\eps}} r^{N-1} \Bigl[ {\textbf{Re}\,} \bigl(e^{-i r x_N} \check w_\eps(rx)\bigr)\cos\bigl((r-1)x_N\bigr) - {\textbf{Im}\,} \bigl(e^{-i r x_N} \check w_\eps(rx)\bigr)\sin\bigl((r-1)x_N\bigr)\Bigr)dr \label{real-part-product}
\end{align}
Here we note that 
$$
|r-1| |x_N| \le \delta_0  \le \frac{\pi}{6} \qquad \text{for $|r-1|\le \sqrt{\eps}$, $x \in M_{\eps,\delta_0}$.}
$$
We thus find that 
\begin{equation}
  \label{eq:1-4-3-4-est}
\cos\bigl((r-1)x_N\bigr) \ge \frac{\sqrt{3}}{2}\quad \text{and}\quad \bigl|\sin\bigl((r-1)x_N\bigr)\bigr| \le \frac{1}{2} \qquad \text{for $|r-1|\le \sqrt{\eps}$, $x \in M_{\eps,\delta_0}$.}
\end{equation}
We now fix $\delta:= \frac{\delta_0}{2}$. By (\ref{eq:real-lower-bound}) and (\ref{Im-modulus-bound}), we have the implications
\begin{align*}
  &x  \in M_{\eps,\delta} \qquad \Longrightarrow \qquad r x \in M_{\eps,\delta_0} \quad \text{for $r \in (0,1+ \sqrt{\eps})$}\qquad \Longrightarrow\\
  &{\textbf{Re}\,} \Bigl(e^{-i x_N} \check w_\eps (rx)\Bigr) \ge  \frac{|C_\eps|}{\sqrt{2}(2\pi)^{\frac{N}{2}}}\quad \text{and}\quad
    {\textbf{Im}\,} \Bigl(e^{-i x_N} \check w_\eps (rx)\Bigr) \le \frac{|C_\eps|}{\sqrt{2}(2\pi)^{\frac{N}{2}}}\qquad \text{for $r \in (0,1+ \sqrt{\eps})$.}
\end{align*}
Combining these estimates with (\ref{real-part-product}) and (\ref{eq:1-4-3-4-est}), we see that
$$
  {\textbf{Re}\,} \Bigl(e^{-i x_N} u_\eps(x)\Bigr) \ge  \int_{1-\sqrt{\eps}}^{1+\sqrt{\eps}} r^{N-1} \frac{|C_\eps|}{\sqrt{2}(2\pi)^{\frac{N}{2}}}\frac{\sqrt{3}-1}{2}dr
\ge  \frac{(\sqrt{3}-1) (1-\sqrt{\eps})^{N-1} \sqrt{\eps}|C_\eps|}{\sqrt{2}(2\pi)^{\frac{N}{2}}} 
$$
for $x \in M_{\eps,\delta}$. This also implies that
$$
\Bigl|u_\eps(x)\Bigr| \ge  \frac{(\sqrt{3}-1) (1-\sqrt{\eps})^{N-1} \sqrt{\eps}|C_\eps|}{\sqrt{2}(2\pi)^{\frac{N}{2}}} \qquad \text{for $x \in M_{\eps,\delta}$.}
$$
From this we deduce that
\begin{equation}
  \label{eq:lower-bound-p-norm-final}
\|u_\eps\|_p \ge \frac{(\sqrt{3}-1)(1-\sqrt{\eps})^{N-1}  \sqrt{\eps}|C_\eps|}{\sqrt{2}(2\pi)^{\frac{N}{2}}} |M_\eps|^{1/p} = \frac{(\sqrt{3}-1)(1-\sqrt{\eps})^{N-1} \eps^{\frac{1}{2}-\frac{N+1}{4p}}|C_\eps|}{\sqrt{2}(2\pi)^{\frac{N}{2}}}.
\end{equation}
By combining (\ref{q-eps-upper-est}) and (\ref{eq:lower-bound-p-norm-final}), we obtain 
$$
R_\eps(p)  \le \frac{\qq(u_\eps)}{\|u_\eps\|_p^2} \le  \frac{2^{N+1}(2\pi)^{N}C|C_\eps|\eps^{\frac{3}{2}}}{(\sqrt{3}-1)^2(1-\sqrt{\eps})^{2(N-1)} \eps^{1-\frac{N+1}{2p}}|C_\eps|^2}  = \frac{2^{N+1}(2\pi)^{N}C}{(\sqrt{3}-1)^2(1-\sqrt{\eps})^{2(N-1)}}\, \frac{\eps^{\frac{1}{2} +\frac{N+1}{2p}}}{|C_\eps|} .
$$
Noting finally that
$$
\frac{2^{N+1}(2\pi)^{N}C}{(\sqrt{3}-1)^2 (1-\sqrt{\eps})^{2(N-1)}} = O(1) \qquad \text{as $\eps \to 0^+$}
$$
and that 
$$
|C_\eps| = \eps^{\frac{N-1}{4}}\bigl( \omega_{N-2}+ o(1)\bigr) \qquad \text{as $\eps \to 0^+$,}
$$
where $\omega_{N-2}$ is the measure of the $N-2$-dimensional unit sphere, we conclude that
$$
R_\eps(p) = O\bigl(\eps^{\frac{1}{2} +\frac{N+1}{2p}-\frac{N-1}{4}}\bigr) = O\bigl(\eps^{\frac{3}{4}+\frac{1}{2p} - \frac{N}{2}(\frac{1}{2}-\frac{1}{p})}\bigr) \qquad \text{as $\eps \to 0^+$.}
$$ 
The claim follows.
\end{proof}

\section{Lower estimates for $R_\eps(p)$ and $R_\eps^{rad}(p)$}
\label{sec:lower-estim-kapp}

We now turn to deriving lower estimates for $R_\eps(p)$ and $R_\eps^{rad}(p)$. We can summarize our results as follows.

\begin{proposition}
\label{prop:lower-estim-kapp}
The following lower bounds hold.  
\begin{enumerate}
\item[(i)]  If $2_* \le p \le 2^*$, we have 
$$
R_\eps(p) \ge 
\frac{2 \CC_{ST}(p)}{\pi}\sqrt{\eps} + o(\sqrt{\eps}) \qquad \text{as $\eps \to 0^+$.}  
$$
\item[(ii)]  If $2_*^{rad} < p \le 2^*$, we have 
$$
R^{rad}_\eps(p) \ge 
\frac{2 \CC_{ST}^{rad}(p)}{\pi} \sqrt{\eps}+  o(\sqrt{\eps}) \qquad \text{as $\eps \to 0^+$.}    
$$

\item[(iii)] If $2< p < 2_*$, we have
$$
R_\eps(p) \ge  
\Bigl(\frac{2 \CC_{ST}(2_*)}{\pi}\Bigr)^{(N+1)(\frac{1}{2}-\frac{1}{p})}
\eps^{\frac{3}{4}+ \frac{1}{2p} -\frac{N}{2}(\frac{1}{2}-\frac{1}{p})} +o(\eps^{\frac{3}{4}+ \frac{1}{2p} -\frac{N}{2}(\frac{1}{2}-\frac{1}{p})}) \qquad \text{as $\eps \to 0^+$.}
$$
\item[(iv)] If $2< p \le 2_*^{rad}$, then we have
  \begin{equation}
    \label{last-lower-bound}
R^{rad}_\eps(p) \ge \Bigl(\frac{2 \CC_{ST}^{rad}(q_\beta)}{\pi}\Bigr)^{2-2\beta}\eps^{\beta} + o(\eps^\beta) \qquad \text{as $\eps \to 0^+$}
  \end{equation}
for every
\begin{equation}
  \label{eq:beta-cond}
  \beta  \:\in\: \left \{
    \begin{aligned}
&\bigl(1- N(\frac{1}{2}-\frac{1}{p})\,,\, \frac{1}{2}+\frac{1}{p} \bigr)&&\qquad \text{in the case $N \le 4$,}\\       
&\bigl(1- N(\frac{1}{2}-\frac{1}{p})\,,\, 1- \frac{N}{4}(\frac{1}{2}-\frac{1}{p})  \bigr)&&\qquad \text{in the case $N \ge 5$}
\end{aligned}
\right.
\end{equation}
and $q_\beta=\frac{4(1-\beta)}{1+\frac{2}{p}-2\beta} \in (2_*^{rad},2^*]$. 
\end{enumerate}
\end{proposition}

\begin{proof}
  (i) Let $\delta \in (0,1)$ and $A_\delta:= \{\xi \in \R^N\::\: \bigl| |\xi|-1\bigr| \le \delta\}$, and let $u \in H$ be a function with $\hat u(\xi) =0$ for $\xi \in \R^N \setminus A_\delta$. Then
\begin{align*}
\|u\|_{p} &= (2\pi)^{-N/2} \Bigl\| \int_{\R^N} e^{i (\cdot) \xi} \hat u(\xi)\,d\xi \Bigr\|_p= (2\pi)^{-N/2} \Bigl\| \int_{1-\delta}^{1+\delta}r^{N-1} \int_{S} e^{i r (\cdot) \theta} \hat u(r \theta)\,d \sigma(\theta) dr\Bigr\|_p\\
&\le (2\pi)^{-N/2} \int_{1-\delta}^{1+\delta} r^{N-1} \Bigl\|  \int_{S} 
e^{i r (\cdot) \theta} \hat u(r \theta)\,d \sigma(\theta) \Bigr\|_p dr \\
& = (2\pi)^{-N/2} \int_{1-\delta}^{1+\delta} r^{N-1-\frac{N}{p}} \Bigl\|  \int_{S}  e^{i (\cdot) \theta} \hat u(r \theta)\,d \sigma(\theta) \Bigr\|_p dr\\
&\le \frac{1}{\sqrt{\CC_{ST}(p)}} \int_{1-\delta}^{1+\delta} r^{N-1-\frac{N}{p}} \Bigl\| \hat u(r (\cdot))\Bigr\|_{L^2(S)} dr\\
&\le \frac{1}{\sqrt{\CC_{ST}(p)}} \Bigl(\int_{1-\delta}^{1+\delta} r^{N-1-\frac{2N}{p}} {g_{\eps}}^{-1}(r)dr\Bigr)^{\frac{1}{2}} \Bigl(\int_{1-\delta}^{1+\delta} r^{N-1} {g_{\eps}}(r) \bigl\| \hat u(r (\cdot))\bigr\|_{L^2(S)}^2 dr\Bigr)^{\frac{1}{2}}\\
&= \frac{1}{\sqrt{\CC_{ST}(p)}} \Bigl(\int_{1-\delta}^{1+\delta} r^{N-1-\frac{2N}{p}}{g_{\eps}}^{-1}(r)dr\Bigr)^{\frac{1}{2}} \Bigl(\int_{1-\delta}^{1+\delta} r^{N-1} {g_{\eps}}(r) \int_{S}|\hat u(r \theta)|^2 d \sigma(\theta) dr\Bigr)^{\frac{1}{2}}\\
&= \frac{1}{\sqrt{\CC_{ST}(p)}} \Bigl(\int_{1-\delta}^{1+\delta} r^{N-1-\frac{2N}{p}}{g_{\eps}}^{-1}(r)dr\Bigr)^{\frac{1}{2}} \Bigl(\int_{\R^N}{g_{\eps}}(|\xi|) |\hat u(\xi)|^2 d\xi \Bigr)^{\frac{1}{2}}\\
&= \frac{1}{\sqrt{\CC_{ST}(p)}} \Bigl(\int_{1-\delta}^{1+\delta} r^{N-1-\frac{2N}{p}}{g_{\eps}}^{-1}(r)dr\Bigr)^{\frac{1}{2}}\sqrt{\qq_{\eps}(u)}\\
&\le \frac{1}{\sqrt{\CC_{ST}(p)}} (1+\delta)^{\frac{N-1}{2}-\frac{N}{p}} \Bigl(\int_{1-\delta}^{1+\delta} {g_{\eps}}^{-1}(r)dr\Bigr)^{\frac{1}{2}}\sqrt{\qq_{\eps}(u)}.
\end{align*}
Here we used the fact that $\frac{N-1}{2}-\frac{N}{p}>0$ since $p>2_*$. Consequently, 
\begin{equation}
\label{lower-est-restricted}  
\frac{\qq_{\eps}(u)}{\|u\|_p^2} \ge (1+\delta)^{\frac{2N}{p}-(N-1)} \CC_{ST}(p)\Bigl(\int_{1-\delta}^{1+\delta} {g_{\eps}}^{-1}(r)dr\Bigr)^{-1}.
\end{equation}
For $\eps>0$, let $u_\eps \in H$ be a function with $\|u_\eps \|_p = 1$ and $\qq_\eps(u_\eps)=R_\eps(p)$, i.e., $u_\eps$ minimizes the quotient in (\ref{eq:def-kappa-eps}). We write $u_\eps = v_\eps + z_\eps$ with 
$$
\hat v_\eps = \hat u_\eps 1_{A_\delta} \qquad \text{and}\qquad \hat z_\eps = 
\hat u_\eps 1_{\R^N \setminus A_\delta} 
$$
By the properties of $g_\eps$, there exists a constant $c= c(\delta)>0$ with 
$$
g_\eps(r) \ge c(1+ r^4) \qquad \text{for $r \in [0,1-\delta] \cup [1+\delta,\infty)$ and $\eps>0$.}
$$
Consequently, by Proposition~\ref{prop:upper-estim-kapp-large} and Sobolev embeddings, 
$$
O(\sqrt{\eps}) \ge  \qq_\eps(u_\eps) \ge \qq_\eps(z_\eps) \ge c \int_{\R^N} (1+ |\xi|^4) |\hat{z}_\eps(\xi)|^2 \ge c_1 \|z_\eps\|_p^2 \qquad \text{as  $\eps \to 0$}
$$ 
with a constant $c_1>0$. Therefore, 
$$
\|v_\eps\|_p \ge \|u_\eps\|_p - \|z_\eps\|_p = 1-o(1) \qquad \text{as $\eps \to 0$.}
$$
Applying the estimate~(\ref{lower-est-restricted}) to $v_\eps$ in place of $u$, we find
$$
R_\eps(p) = \qq_{\eps}(u_\eps) \ge \qq_{\eps}(v_\eps)= (1-o(1))\frac{\qq_\eps(v_\eps)}{\|v_\eps\|_p^2}= (1-o(1))(1+\delta)^{\frac{2N}{p}-(N-1)} \CC_{ST}(p) \Bigl(\int_{1-\delta}^{1+\delta} {g_{\eps}}^{-1}(r)dr\Bigr)^{-1}.
$$
From Lemma~\ref{lemma-appendix} below, we thus deduce that 
\begin{align*}
\liminf_{\eps \to 0} \frac{R_\eps(p)}{\sqrt{\eps}} & \ge  (1+\delta)^{\frac{2N}{p}-(N-1)} \CC_{ST}(p) \lim_{\eps \to 0} \frac{1}{\sqrt{\eps}}\Bigl(\int_{1-\delta}^{1+\delta} {g_{\eps}}^{-1}(r)dr\Bigr)^{-1} \\
& = (1+\delta)^{\frac{2N}{p}-(N-1)} \frac{2 \CC_{ST}(p)}{\pi} .
\end{align*}
Since $\delta \in (0,1)$ can be chosen arbitrarily, we conclude that 
$$
\liminf_{\eps \to 0} \frac{R_\eps(p)}{\sqrt{\eps}} \ge  \frac{2 \CC_{ST}(p)}{\pi}. 
$$
(ii) This follows by precisely the same argument as in (i), where now the definition of $\CC_{ST}^{rad}(p)$ is used in place of Theorem~\ref{stein-tomas}.\\
(iii) We use the interpolation inequality  
$$
\|u\|_p \le \|u\|_2^{1-\alpha}\|u\|_{2_*}^{\alpha} \qquad \text{for $u \in H$}
$$
with
$$
\alpha = \frac{1-\frac{2}{p}}{1-\frac{2}{2_*}} = \frac{1-\frac{2}{p}}{1-\frac{N-1}{N+1}}= (N+1)\bigl(\frac{1}{2}-\frac{1}{p}\bigr).
$$
Now for every $u \in H$ we have 
$$
\eps \|u\|_2^2  = \eps \int_{\R^N}|\hat u (\xi)|^2d\xi \le \int_{\R^N}g_\eps(|\xi|)|\hat u (\xi)|^2d\xi =  \qq_\eps(u), 
$$
and therefore 
$$
\frac{\qq_\eps(u)}{\|u\|_p^2} \ge \frac{\qq_\eps(u)^{1-\alpha} \qq_\eps(u)^\alpha }{\bigl(\|u\|_2^2\bigr)^{1-\alpha}\bigl(\|u\|_{2_*}^2\bigr)^{\alpha}} \ge \eps^{1-\alpha} \Bigl(\frac{\qq_\eps(u)}{\|u\|_{2_*}^2}\Bigr)^{\alpha}.
$$
Consequently, by (i),
\begin{align*}
  R_\eps(p) \ge \eps^{1-\alpha} \Bigl( \inf_{u \in H \setminus \{0\}}\frac{\qq_\eps(u)}{\|u\|_{2_*}^2}\Bigr)^{\alpha} 
  &\ge \eps^{1-\alpha} \Bigl(\frac{2 \CC_{ST}(2_*)}{\pi}\eps^{\frac{1}{2}}+o(\eps^{\frac{1}{2}})\Bigr)^{\alpha}\\
  &\ge \Bigl(\frac{2 \CC_{ST}(2_*)}{\pi}\Bigr)^\alpha \eps^{1-\frac{\alpha}{2}} + o(\eps^{1-\frac{\alpha}{2}})
\end{align*}
Since $1-\frac{\alpha}{2} = \frac{3}{4}+ \frac{1}{2p} -\frac{N}{2}(\frac{1}{2}-\frac{1}{p})$, the claim follows.\\
(iv) We argue similarly as in (iii), choosing now $q> 2_*^{rad}=\frac{2N}{N-1}$ and setting $\alpha = \alpha_{p,q}= \frac{1-\frac{2}{p}}{1-\frac{2}{q}}$. We then use the interpolation inequality $\|u\|_p \le \|u\|_2^{1-\alpha}\|u\|_{q}^{\alpha}$ and obtain, as above, that 
$$
\frac{\qq_\eps(u)}{\|u\|_p^2} \ge \eps^{1-\alpha} \Bigl(\frac{\qq_\eps(u)}{\|u\|_{q}^2}\Bigr)^{\alpha} \qquad \text{for every $u \in H_{rad}$.}
$$
Consequently, by (ii), 
\begin{align}
R_\eps^{rad}(p) \ge \eps^{1-\alpha} \Bigl( \inf_{u \in H_{rad} \setminus \{0\}}\frac{\qq_\eps(u)}{\|u\|_{q}^2}\Bigr)^{\alpha} 
  &\ge \eps^{1-\alpha} \Bigl(\frac{2 \CC_{ST}^{rad}(q)}{\pi}\eps^{\frac{1}{2}}+o(\eps^{\frac{1}{2}})\Bigr)^{\alpha} \nonumber\\
&\ge \Bigl(\frac{2 \CC_{ST}^{rad}(q)}{\pi}\Bigr)^\alpha \eps^{1-\frac{\alpha}{2}} + o(\eps^{1-\frac{\alpha}{2}}). \label{last-lower-bound-proof}
\end{align}
Now, for any $\beta$ satisfying the restrictions in (\ref{eq:beta-cond}), we now choose 
$$
q = q_\beta = \frac{2}{1- \frac{1-\frac{2}{p}}{2-2\beta}}= \frac{4(1-\beta)}{1+\frac{2}{p}-2\beta} \quad \in \quad \bigl(2_*^{rad}, 2^*\bigr),  
$$
and we note that
$$
\alpha = \alpha_{p,q} = \frac{1-\frac{2}{p}}{1-\frac{2}{q}}= 2- 2\beta
$$
in this case, i.\,e, we have $1-\frac{\alpha}{2} = \beta$. Hence (\ref{last-lower-bound-proof}) implies (\ref{last-lower-bound}). The proof of Proposition \ref{prop:lower-estim-kapp} is thus finished.
\end{proof}

\section{The Dirichlet problem in the unit ball}
\label{sec:dirichl-probl-unit}
In this section we complete the proof of Theorem~\ref{theorem-nonradial-ball}. For this we fix $p \in (2,2_*)$ and we put $B_r:= B_r(0)$ for $r>0$. Moreover, we consider, for $\eps>0$, the restriction
$$
u \mapsto \qq_{\eps}^r(u) = \int_{B_r}\Bigl( |\Delta u|^2- 2   |\nabla u|^2 + (1+\eps) |u|^2 \Bigr)\,dx 
$$
of the quadratic form $\qq_{\eps}$ defined in \eqref{eq:def-q-eps} to the subspace $H_0^2(B_r) \subset H$. We also define
\begin{equation}
  \label{eq:def-kappa-eps-r}
R_\eps^r(p):= \inf_{u \in H_0^2(B_r) \setminus \{0\} } \frac{\qq_\eps^r(u) }{ \|u\|_{L^p(B_r)}^2 }
\end{equation}
and we note that $R_\eps^r(p) \ge R_\eps(p)$ for every $r>0$. Moreover, from the fact that $C_0^\infty(\R^N)$ is dense in $H$ and Sobolev embeddings, it is easy to deduce that
\begin{equation}
  \label{eq:R-eps-limit}
R_\eps^r(p) \to R_\eps(p) \qquad \text{as $r \to \infty$.}
\end{equation}
We also define 
\begin{equation*}
R_\eps^{rad,r}(p):= \inf_{\stackrel{u \in H_0^2(B_r) \setminus \{0\}}{u \text{ radial}} } \frac{\qq_\eps^r(u) }{ \|u\|_{L^p(B_r)}^2 },
\end{equation*}
and we note that
\begin{equation}
  \label{eq:R-eps--rad-ineq}
R_\eps^{rad,r}(p) \ge R_\eps^{rad}(p) \qquad \text{for every $r>0$,}
\end{equation}
where $R_\eps^{rad}(p)$ is defined in (\ref{eq:def-kappa-eps-rad}). Next, we let $0 < \eps \le \eps_0$ with $\eps_0=\eps_0(p)$ given in Theorem~\ref{theorem-nonradial}. Combining (\ref{eq:rad-strict-ineq}) with (\ref{eq:R-eps-limit}) and (\ref{eq:R-eps--rad-ineq}), we find that there exists $a_0= a_0(p,\eps)>0$ with the property that
\begin{equation}
  \label{eq:R-rad-eps-ineq-2}
  R_\eps^{rad,r}(p) >R_\eps^{r}(p)\qquad \text{for $r>\sqrt{a_0}$.}
\end{equation}
Next, we define, for $a>0$, the scaling map $u \mapsto \cS_a u:= u( \frac{(\cdot)}{\sqrt{a}})$, which is a topological isomorphism between the spaces $H^2_0(B)$ and $H^2_0(B_{\sqrt{a}})$. Moreover, $\cS_a$ maps radial functions to radial functions. A change of variables shows that 
\begin{equation}
\label{scaling-eq}  
\qq_{\eps}^{\sqrt{a}}(\cS_a u) = a^{\frac{N}{2}-2} \qq_{a,(1+\eps)a^2,B}(u)\quad \text{and}\quad \|\cS_a u\|_{L^p(B_{\sqrt{a})})} = a^{\frac{N}{2p}} \|u\|_{L^p(B)}
\end{equation}
for every $u \in H^2_0(B)$, where $\qq_{a,(1+\eps)a^2,B}$ is defined in (\ref{eq:def-kappa-a-b-ball}). By (\ref{eq:R-rad-eps-ineq-2}) and (\ref{scaling-eq}), we have
$$
\frac{\qq_{a,(1+\eps)a^2,B}(u)}{\|u\|_{L^p(B)}^2} > R_{a,(1+\eps)a^2,B} \qquad \text{for every $a>a_0$ and every radial function $u \in H^2_0(B)$,}
$$
where $R_{a,(1+\eps)a^2,B}$ is defined in (\ref{eq:def-kappa-a-b-ball}).
Hence, if $a>a_0$ and $b= (1+\eps)u^2$, then every ground state solution of  
 (\ref{eq:biharmonic-bvp-ball}) is nonradial. This concludes the proof of Theorem~\ref{theorem-nonradial-ball}.

\begin{appendix}

\section{A Technical Lemma}
\label{sec:technical-lemma}

\begin{lemma}
\label{lemma-appendix}
Suppose that $a \in \R$, $\delta>0$, and let $g_0:[a-\delta,a+\delta]$ be a $C^2$-function with $g_0(r)>0$ for $r \in [a-\delta,a+\delta] \setminus \{a\}$ and $g_0(a)=g_0'(a)=0$, $g_0''(a)>0$. Moreover, let 
$$
g_\eps:[a-\delta,a+\delta] \to \R, \qquad g_\eps(r)=g_0(r)+\eps.
$$
Finally, let $\eps_0>0$ and $\tau: (0,\eps_0) \to [0,\delta]$ be a function with
$\lim \limits_{\eps \to 0^+}\frac{\tau(\eps)}{\sqrt{\eps}}= \infty.$
Then we have 
$$
\lim_{\eps \to 0^+} \sqrt{\eps} \int_{a-\tau(\eps)}^{a+\tau(\eps)} \frac{dr}{g_\eps(r)} = 
\frac{\pi}{\sqrt{\frac{1}{2}g_0''(a)}}.
$$
\end{lemma}

\begin{remark}{\em 
  \label{remark-appendix}
  Under the assumptions of Lemma~\ref{lemma-appendix}, it follows in particular, by choosing the constant function $\tau \equiv \delta$, that
  $$
  \lim_{\eps \to 0^+} \sqrt{\eps} \int_{a-\delta}^{a+\delta} \frac{dr}{g_\eps(r)} = \frac{\pi}{\sqrt{\frac{1}{2} g_0''(a)}}
 $$
 Moreover, it follows that
 $$
\lim_{\eps \to 0^+} \sqrt{\eps} \int_{a-\eps^s}^{a+\eps^s} \frac{dr}{g_\eps(r)} = 
\frac{\pi}{\sqrt{\frac{1}{2} g_0''(a)}} \qquad \text{for every $s \in (0,\frac{1}{2})$.}
$$ }
\end{remark}

\begin{proof}[Proof of Lemma~\ref{lemma-appendix}]
Without loss of generality, we assume that $a=0$. By assumption, there exists a constant $c>0$ with 
$$
g_0(r) \ge c r^2 \qquad \text{for $r \in [-\delta,\delta].$}
$$
Then 
\begin{align*}
\sqrt{\eps} \int_{-\tau(\eps)}^{\tau(\eps)} \frac{dr}{g_\eps(r)} &= \sqrt{\eps} 
\int_{-\tau(\eps)}^{\tau(\eps)} \frac{dr}{g_0(r)+\eps} 
= \int_{-\frac{\tau(\eps)}{\sqrt{\eps}}}^{\frac{\tau(\eps)}{\sqrt{\eps}}} 
\frac{\eps \, dr}{g_0(\sqrt{\eps} r)+\eps}  = \int_{-\frac{\tau(\eps)}{\sqrt{\eps}}}^{\frac{\tau(\eps)}{\sqrt{\eps}}} 
\frac{dr}{\frac{g_0(\sqrt{\eps} r)}{\eps}+1} 
\end{align*}
for every $\eps>0$ and  
$$
\frac{1}{\frac{g_0(\sqrt{\eps} r)}{\eps}+1} \le \frac{1}{c r^2 +1}\qquad \text{for $\eps >0,$ $r \in [-\frac{\tau(\eps)}{\sqrt{\eps}},\frac{\tau(\eps)}{\sqrt{\eps}}]$.}
$$
Since also 
$$
\lim_{\eps \to 0^+} \frac{1}{\frac{g_0(\sqrt{\eps} r)}{\eps}+1}=\frac{1}{\frac{1}{2} g_0''(0) r^2 +1} \qquad \text{for all $r \in \R$,}
$$
Lebesgue's theorem implies that 
$$
\sqrt{\eps} \int_{-\tau(\eps)}^{\tau(\eps)} \frac{dr}{g_\eps(r)} \to \int_{\R} \frac{dr}{\frac{1}{2} g_0''(0) r^2 +1} =  \frac{\pi}{\sqrt{\frac{1}{2} g_0''(0)}} \qquad \text{as $\eps \to 0^+$.}
$$
\end{proof}

\section{Existence and Properties of Ground States}
\label{sec:exist-prop-ground}

In this section, we provide the proof of Theorem \ref{sec:prop-intro}. Assume that $b > a^2$ and $2 < p < 2^*$. Let us first prove that the infimum $R_{a,b}(p)$ is attained. Suppose that $(u_n) \subset H^2(\R^N)$ is a sequence with $\|u_n\|_p= 1$ for all $n \in \N$ and
  \begin{equation}
    \label{eq:c-eps-approx}
\qq_{a,b}(u_n) \to R_{a,b}(p) \quad \mathrm{as} \quad n \to \infty.
  \end{equation}
In particular, the sequence $(u_n)$ is bounded in $H^2(\R^N)$. Since we must have $u_n \not \to 0$ in $L^p(\R^N)$, it follows from Lions' concentration compactness lemma for bounded sequences in $H^2(\R^N)$ together with $2 < p < p_*$ that there exist points $z_n \in \R^N$ with the property that, after passing to a subsequence, the sequence of functions $u_n (\cdot - z_n)$ has a nontrivial weak limit $u \in H^2(\R^N) \setminus \{0 \}$, say.  By translation invariance, we can replace $u_n$ by $u_n(\cdot - z_n)$ for every $n \in \N$, so that
  $$
  u_n \weak u \in  H^2(\R^N) \setminus \{0\}.
  $$
  By Fatou's lemma, we have
  $$
  0 < \|u\|_{p} \le \liminf_{n \to \infty} \|u_n\|_p.
  $$
  We claim that $u_n \to u \in L^p(\R^N)$. If this is not the case, then we may pass to a subsequence with the property that $\|v_n\|_p \to d > 0$ where we set $v_n = u-u_n$. Then we have $\|u\|_p^p + d^p= 1 $ and thus $\|u\|_p^2 + d^2 >1$. Consequently, we deduce
  \begin{align*}
    R_{a,b}(p) &= \lim_{n \to \infty} \qq_{a,b}(u_n) = \lim_{n \to \infty} \Bigl(\qq_{a,b}(u) + \qq_{a,b}(v_n)\Bigr)\\
                &\ge R_{a,b}(p) \lim_{n \to \infty} \Bigl(\|u\|_p^2 + \|v_n\|_p^p\Bigr) = R_{a,b}(p) \Bigl(\|u\|_p^2 + d^p\Bigr)>R_{a,b}(p),
  \end{align*}
  which is contradiction. Hence $u_n \to u$ in $L^p(\R^N)$, which implies that $\|u\|_{L^p}=1$ and, by weak lower semicontinuity,
  $$
 \qq_{a,b}(u) \le \lim_{n \to \infty}\qq_{a,b}(u_n) = R_{a,b}(p).
  $$
  Therefore the infimum $R_{a,b}(p)$ is attained at $u \in H^2(\R^N)$ with $u \not \equiv 0$.   

Next, we claim the any minimizer $u \in H^2(\R^N) \setminus \{0\}$ for $R_{a,b}(p)$ must be real-valued up to a trivial constant complex phase, i.\,e., there is some constant $\theta \in \R$ such that
$$
e^{i \theta} u(x) \in \R \quad \mbox{for a.\,e.~$x \in \R^N$}.
$$
To prove this claim, we adapt an argument in \cite{FrLiSa-2016} (see also \cite{BuLeSo-2019}). First, we recall that
$$
\qq_{a,b}(u) = \int_{\R^N} g_{a,b}(|\xi|) |\hat{u}(\xi)|^2 \, d \xi
$$
where $g_{a,b}(|\xi|) = |\xi|^4 - 2 a |\xi|^2 + b$. We split $u : \R^N \to \C$ into real and imaginary part so that 
$$
u(x) = u_R(x) + i u_I(x)
$$ 
with the real-valued functions $u_R, u_I : \R^N \to \R$. From now on, we assume that $u_R \not \equiv 0$ and $u_I \not \equiv 0$ are both non-trivial, since otherwise the result would directly follow. By elementary properties of the Fourier transform, it holds that $\hat{u}_R(-\xi) = \overline{\hat{u}_R(\xi)}$ and $\hat{u}_I(-\xi) = \overline{\hat{u}_I(\xi)}$ for a.\,e.~$\xi \in \R^N$. Using that $g_{a,b}(|\xi|)$ is a real-valued and even, an elementary calculation shows 
\be \label{eq:q_split}
\qq_{a,b}(u) = \qq_{a,b}(u_R) + \qq_{a,b}(u_I).
\ee
On the other hand, we use the fact that $p > 2$ to deduce that
\be \label{ineq:hanner}
\| u \|_{L^p}^2 = \| |u_R|^2 + |u_I|^2 \|_{L^{p/2}} \leq \| |u_R|^2 \|_{L^{p/2}} + \| |u_I|^2 \|_{L^{p/2}} = \| u_R \|_{L^p}^2 + \| u_I \|_{L^p}^2 .
\ee
From the strict convexity of the $L^{p/2}$-norm for $p> 2$ (Hanner's inequality) we see that equality holds in \eqref{ineq:hanner} if and only if $u_I=0$ or $u_R^2 = \alpha^2 u_I^2$ with some constant $\alpha \geq 0$. 

By using \eqref{eq:q_split} and \eqref{ineq:hanner} together with the fact that $u \in H$ minimizes $R_{a,b}(p)$, we find
$$
R_{a,b}(u) = \frac{\qq_{a,b}(u)}{\|u \|_{L^p}^2} \geq \frac{\qq_{a,b}(u_R) + \qq_{a,b}(u_I)}{\| u_R \|_{L^p}^2 + \| u_I \|_{L^p}^2} \geq \min \left ( \frac{\qq_{a,b}(u_R)}{\| u_R \|_{L^p}^2}, \frac{\qq_{a,b}(u_I)}{\| u_I \|_{L^p}^2} \right ) \geq R_{a,b}(p).
$$
Since $u_R \not \equiv 0 \not \equiv u_I$ by assumption, equality in \eqref{ineq:hanner} yields that $u_I^2 = \alpha^2 u_R^2$ for some constant $\alpha > 0$. Next we claim that $u_I = \pm \alpha u_R$ holds, which would complete the proof. To see this, we define the real-valued functions $u_1 = \frac{1}{\sqrt{2}}(u_R + u_I)$ and $u_2 = \frac{1}{\sqrt{2}}(-u_R+u_I)$. This gives us
$$
u(x) = e^{i\pi/4} u_1(x) + i e^{i \pi/4} u_2(x).
$$
If $u_2 \equiv 0$ or $u_1 \equiv 0$, we are done since this implies $u_I = \pm u_R$. Thus we can assume $u_1 \not \equiv 0$ and $u_2 \not \equiv 0$.  In the same fashion as above, find that we can use \eqref{eq:q_split} and \eqref{ineq:hanner} with $u_R, u_I$ replaced by $u_1, u_2$.  Hence we deduce that $u_2^2 = \beta^2 u_1^2$ with some constant $\beta > 0$. Notice that $\beta^2 \neq 1$ holds, because otherwise we get $u_R u_I \equiv 0$ (which yields $u \equiv 0$ from $u_I = \alpha^2 u_R^2$). In summary, we have found that
$$
u_R^2 = \alpha^2 u_I^2 \quad \mbox{and} \quad \frac{1}{2} ( 1+ \alpha^2) (1-\beta^2) u_R^2 = (1+\beta^2) u_R u_I,
$$
which implies that $u_I = \pm \alpha u_R$. Thus the functions $u_R$ and $u_I$ are linearly dependent and hence $e^{i \theta} u(x) \in \R$ almost everywhere in $\R^N$ with some constant $\theta \in \R$. This completes the proof of Theorem \ref{sec:prop-intro}. \hfill $\square$

\medskip
Finally, we use the method of Fourier symmetrization to obtain the following symmetry result in the case of even-integer $p > 2$.

\begin{lemma} \label{lem:even}
Let $N \geq 1$, $2 < p < 2^*$, $b > a^2$, and suppose that $p \in 2 \mathbb{N}$ is an even integer. Then any minimizer $u \in H$ for $R_{a,b}(p)$ must be an even function, i.\,e., we have $u(-x) = u(x)$ for a.\,e.~$x \in \R^N$.
\end{lemma}

\begin{rmk*}
{\em From Theorem \ref{theorem-nonradial} we recall that breaking of radial symmetry for minimizers $u \in H$ must occur in the range $N \geq 2$ and $2 < p < 2_*$. In this range, the only compatible choice with Lemma \ref{lem:even} is $N=2$ and $p=4 < 2_*$. }
\end{rmk*}

\begin{proof}
We can invoke the general strategy developed in \cite{BuLeSo-2019} based on Fourier methods. For the reader's convenience, we provide some details on how to apply these results to our setting. In \cite{BuLeSo-2019}, one consider minimizers of functionals of the form $J : H^s(\R^N) \setminus  \{ 0 \} \to \R$ with
$$
J(f) = \frac{\langle f, (P(D) + \lambda) f\rangle}{\| f \|_{L^{2 \sigma +2}}^2}.
$$ 
Here $P(D)$ is an elliptic (pseudo-)differential operator of order $2s > 0$ with Fourier symbol $p(\xi)$ and $0 < \sigma < \sigma_*$ with $\sigma_* = \frac{2s}{N-2s}$ if $s < N/2$ and $\sigma_* = \infty$ if $s \geq N/2$. The constant $\lambda \in \R$ is assumed to satisfy 
$$
\inf_{\xi \in \R^N} p(\xi) + \lambda > 0,
$$ 
which guarantees the norm equivalence $\| f \|_{H^s}^2 \simeq \langle f, (P(D) + \lambda) f \rangle$. Adapted to our setting, we have $s=2$ and
$$
P(D) = \Delta^2 + 2 a\Delta, \quad \lambda = b, \quad \sigma = \frac{p-2}{2}.
$$ 
Notice that $\inf_{\xi \in \R^N} p(\xi)+\lambda > 0$ is equivalent to $b > a^2$. Furthermore, it is elementary to check that the condition $0 < \sigma < \sigma_*$ is equivalent to $2 < p < 2_*$. Thus the minimizers $u \in H$ for $R_{a,b}(p)$ are exactly minimizers of the functional $J(f)$ above with the our choice of $P(D), \lambda, \sigma$ above. 

If $p \in 2 \N$ is an even integer (and thus $\sigma = \frac{1}{2} (p-2) \in \N$ is an integer), we can apply \cite{BuLeSo-2019}[Theorem 2] to deduce that any minimizer $u \in H$ for $R_{a,b}(p)$ must be an even function, i.\,e., it holds that
$$
u(-x) = u(x) \quad \mbox{for a.\,e.~$x \in \R^N$},
$$
using also that the symbol $p(\xi) = |\xi|^4 -2 a |\xi|^2$ is real-valued and even together with the fact that we have the exponential decay property
$$
e^{\mu |\cdot |} u \in L^2(\R^N)
$$
for some $\mu > 0$. In fact, the later exponential decay property can be deduced from Paley--Wiener type arguments and complex continuation arguments, which are classical for the operator $P(D)=\Delta^2 + 2a \Delta$ having an analytic symbol; see also the remark following \cite{BuLeSo-2019}[Theorem 2] for exactly this example for $P(D)$.

This completes the sketch of the proof of Lemma \ref{lem:even}. \end{proof} 

{\bf Acknowledgment:} The second author would like to thank Antonio Fern\'andez for valuable comments.

\end{appendix}

\bibliographystyle{plain}

\end{document}